\documentclass[12pt]{article}
\usepackage{latexsym,amssymb,amsfonts,amsmath,amsthm}
\usepackage[hidelinks]{hyperref}
\usepackage[margin=2cm]{geometry}
\usepackage{mdwlist}
\usepackage{calrsfs}
\DeclareMathAlphabet{\pazocal}{OMS}{zplm}{m}{n}

\usepackage{amsrefs}

\usepackage{bm}
\usepackage{longtable}
\usepackage{pdflscape}
\usepackage{rotating}
\usepackage{algorithm}
\usepackage{algorithmic}
\usepackage[all]{xy}
\usepackage{comment}
\usepackage{url}
\usepackage{CJKutf8}
\usepackage{graphicx}
\usepackage[usenames,dvipsnames]{xcolor}
\usepackage{caption}
\usepackage{subcaption}
\usepackage{tikz}
\usetikzlibrary{math}

\newtheorem{theorem}{Theorem}[section]
\newtheorem{proposition}[theorem]{Proposition}

\newtheorem{conjecture}[theorem]{Conjecture}
\theoremstyle{definition}
\newtheorem{definition}[theorem]{Definition}

\newcommand{\CA}{\text{CA}}

\newcommand{\CAFE}{\text{CAFE}}

\numberwithin{equation}{section}

\let\oldproofname=\proofname
\renewcommand{\proofname}{\rm\bf{\oldproofname}}

\makeatletter
\def\imod#1{\allowbreak\mkern10mu({\operator@font mod}\,\,#1)}
\makeatother

\newcommand{\ignore}[1]{}

 {\begin{list}{}%
  {%
   \settowidth{\labelwidth}{\textnormal{#1 =}}%
   \setlength{\leftmargin}{\labelwidth}%
   \addtolength{\leftmargin}{\labelsep}%
   \setlength{\itemsep}{-\parsep}}}%
 {\end{list}}
\begin{document}

\title{Non-uniform covering array with symmetric forbidden edge constraints}

\author{Brett Stevens\\
School of Mathematics and Statistics\\
Carleton University\\
1125 Colonel By Drive\\
Ottawa, ON, K1S 5B6\\
Canada
\texttt{brett@math.carleton.ca}
}

\date{\today}

\maketitle

\begin{abstract}
  It has been conjectured that whenever an optimal covering array exists there is also a uniform covering array with the same parameters and this is true for all known optimal covering arrays.  When used as a test suite, the application context may have pairs of parameters that must be avoided and Covering arrays avoiding forbidden edges (CAFE) are a generalization accommodating this requirement.  We prove that there is an arc-transitive, highly symmetric constraint graph where the unique optimal covering array avoiding forbidden edges is not uniform.  This does not refute the conjecture but it does show that placing even highly symmetric constraints on covering arrays can force non-uniformity of optimal arrays.
\end{abstract}

A column of a covering array is {\em uniform} if the number of appearances of any two symbols differ by no more than 1.  For all the known optimal covering arrays there exists an optimal covering array with the same parameters that is uniform on every column.  This includes the two infinite families of optimal covering arrays, orthogonal arrays \cite{hedayat_orthogonal_1999} and strength 2 binary covering arrays constructed by Katona and Kleitman and Spencer \cites{katona_two_1973,kleitman_families_1973}, and 21 other known optimal covering arrays of strength 2\cite{KMNNOS_18}.  It has been conjectured that there exists a uniform covering array of optimal size for all parameters \cite{MR2156344} and this conjecture was the motivation for a recent enumeration of small covering arrays \cite{KMNNOS_18}.

An analogous uniformity conjecture for covering and packing (error-correcting) codes has been disproven. For binary covering codes, there are sets of parameters for which all optimal codes are nonuniform \cite{ostergaa_rd_disproof_2003}. Furthermore, for binary error-correcting codes, there are sets of parameters for which all optimal codes have a nonuniform distribution of coordinate values in all coordinates \cite{ostergaa_rd_optimal_2013}.

In this note we prove that the unique optimal covering array avoiding a highly symmetric set of forbidden edge constraints has three of four columns non-uniform and is therefore not uniform itself.  This does not refute the conjecture from \cite{KMNNOS_18} but it does show that placing even highly symmetric constraints on covering arrays can force non-uniformity of optimal arrays.  The uniformity question for other covering array generalizations and constraints is open.

\begin{definition}
A {\em covering array} $\CA(N;t,k,v)$ of strength $t$ is an $N \times k$
array of symbols from $[0,v-1]$ such that in every
$N \times t$ subarray, every $t$-tuple occurs in at least one row. A covering array is \emph{optimal} if it has the
smallest possible $N$ for given $t$, $k$, and $v$, and \emph{uniform}
if every symbol occurs either $\lfloor N/v \rfloor$ or $\lceil N/v \rceil$
times in every column.

Let $G=G_{k,v}$ be a k-partite hypergraph graph with vertex set $V(G) = \{ v_{i,a}:i \in [1,k],\;a \in [0,v-1]\}$. A row of a covering array ${\bf c} \in [0,v-1]^k$ {\em avoids} $G$ if for
all $\{i_1, i_2, \ldots , i_t\} \subset [1,k]$, we have $\{v_{i_1,{\bf c}_{i_1}},\ldots,v_{i_t,{\bf c}_{i_t}}\}\not\in E(G)$. A {\em covering array with forbidden edges} $\CAFE(N;G)$ is a $\CA(N;t,k,v)$ where every row avoids $G$.
\end{definition} For more about $\CAFE$s see \cite{danziger_covering_2009} and for more about covering arrays in general see \cites{hartman_software_2005,colbourn_combinatorial_2004,lawrence_survey_2011}.

Let $G = G_{4,3}$ have edges
\[
  E(G) = \left \{ \{v_{0,0},v_{1,0}\},\{v_{2,0},v_{3,0}\},\{v_{0,1},v_{2,1}\},\{v_{1,1},v_{3,1}\},\{v_{0,2},v_{3,2}\},\{v_{1,2},v_{2,2}\} \right \}
\]
as shown in Figure~\ref{graph_fig}.
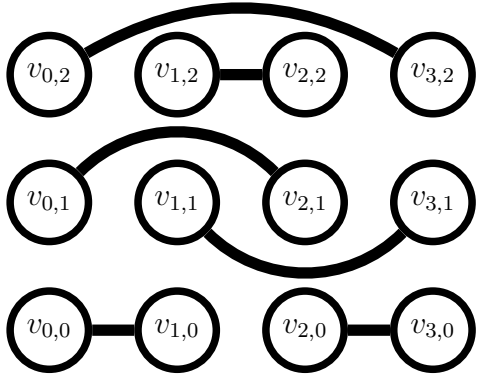
\begin{figure}
  \begin{center}
\begin{tikzpicture}
% the vertices

      \tikzmath{\y1 = 1.7; \y2=1.5; }
      \node (00) at (0,0) [circle,line width=1.0mm,draw=black]  {$v_{0,0}$};
      \node (01) at (0,\y1) [circle,line width=1.0mm,draw=black]  {$v_{0,1}$};
      \node (02) at (0,2*\y1) [circle,line width=1.0mm,draw=black]  {$v_{0,2}$};
      \node (10) at (1*\y1,0) [circle,line width=1.0mm,draw=black]  {$v_{1,0}$};
      \node (11) at (1*\y1,1*\y1) [circle,line width=1.0mm,draw=black]  {$v_{1,1}$};
      \node (12) at (1*\y1,2*\y1) [circle,line width=1.0mm,draw=black]  {$v_{1,2}$};
      \node (20) at (2*\y1,0) [circle,line width=1.0mm,draw=black]  {$v_{2,0}$};
      \node (21) at (2*\y1,1*\y1) [circle,line width=1.0mm,draw=black]  {$v_{2,1}$};
      \node (22) at (2*\y1,2*\y1) [circle,line width=1.0mm,draw=black]  {$v_{2,2}$};
      \node (30) at (3*\y1,0) [circle,line width=1.0mm,draw=black]  {$v_{3,0}$};
      \node (31) at (3*\y1,1*\y1) [circle,line width=1.0mm,draw=black]  {$v_{3,1}$};
      \node (32) at (3*\y1,2*\y1) [circle,line width=1.0mm,draw=black]  {$v_{3,2}$};

      \draw [-,line width=1.5mm,draw=black] (00) to[out=0,in=180] (10);
      \draw [-,line width=1.5mm,draw=black] (20) to[out=0,in=180] (30);
      \draw [-,line width=1.5mm,draw=black] (01) to[out=45,in=135] (21);
      \draw [-,line width=1.5mm,draw=black] (11) to[out=315,in=225] (31);
      \draw [-,line width=1.5mm,draw=black] (02) to[out=30,in=150] (32);
      \draw [-,line width=1.5mm,draw=black] (12) to[out=0,in=180] (22);
    \end{tikzpicture}
    \end{center}
    \caption{The graph $G$.\label{graph_fig}}
  \end{figure}
  The automorphism group of $G$ respecting its partition structure has order 24 and is generated by
  \begin{align*}
  g_1 &=  (v_{0,0} v_{1,0})(v_{0,1}, v_{1,2})(v_{0,2}, v_{1,1})(v_{2,1},v_{2,2})(v_{3,1},v_{3,2}), \\
  g_2 &= (v_{0,0},v_{1,2},v_{2,0},v_{3,2})(v_{1,0},v_{2,2},v_{3,0},v_{0,2})(v_{0,1},v_{1,1},v_{2,1},v_{3,1}).
  \end{align*}
  Under this automorphism group, $G$ is arc-transitive which implies vertex transitive.
  This is a high degree of symmetry.  In particular, all the symbols in the covering array (vertices in the graph) are equivalent.  Despite this symmetry we will show that the unique $\CAFE$ on $G$ is not uniform.  

\begin{proposition}
  A $\CAFE(12,G)$ is optimal.
\end{proposition}
\begin{proof}
  A pair $\{v_{i_1, s_1},v_{i_2,s_2}\}$ with $i_1 \neq i_2$ and $s_1 \neq s_2$ is {\em solitary} if it is disjoint from the unique edge between parts $i_1$ and $i_2$.  The solitary pairs are
  \[
    \begin{array}{rccl}
      \left \{ \{v_{0,1},v_{1,2} \}, \right . & \{v_{0,2},v_{1,1} \}, & \{v_{2,1},v_{3,2} \}, \\
      \{v_{2,2},v_{3,1} \}, & \{v_{0,0},v_{2,2} \}, & \{v_{0,2},v_{2,0} \}, \\
      \{v_{1,0},v_{3,2} \}, & \{v_{1,2},v_{3,0} \}, & \{v_{0,0},v_{3,1} \},  \\
      \{v_{0,1},v_{3,0} \}, & \{v_{1,0},v_{2,1} \}, & \left .\{v_{1,1},v_{2,0} \} \right \}
    \end{array}.
  \]
  We show that no two solitary pairs can be in the same row of a $\CAFE(N;G)$. The set of solitary pairs is an orbit under the automorphism group of $G$, thus we need only prove that the first solitary pair $\{v_{0,1},v_{1,2} \}$ cannot appear in a row with any other solitary pair in a $\CAFE(N;G)$. If ${\bf c}$ is a row with ${\bf c}_0 = 1$ and ${\bf c}_1 = 2$, avoiding the edges of $G$ forces ${\bf c}_2 = 0$.  Thus the only other solitary pairs which ${\bf c}$  could contain are $\{v_{1,2},v_{3,0} \}$ or  $\{v_{0,1},v_{3,0} \}$, but both of these are forbidden from appearing in ${\bf c}$ by the edge $\{v_{2,0},v_{3,0} \} \in E(G)$.
    \end{proof}

    \begin{theorem}
      Up to isomorphism there is a unique $\CAFE(12;G)$ that has a single uniform column and three non-uniform columns.
      \end{theorem}
      \begin{proof}
    If there is a $\CAFE(12;G)$ then there is one row for each solitary pair.  A row with a solitary pair avoiding $G$ forces one additional symbol in that row. Thus the array must have form given in Table~\ref{cafe1}.
    \begin{table}
      \begin{center}
      \begin{tabular}{|cccc|}
        \hline
        1 & 2 & 0 &   \\
        2 & 1 &   & 0 \\
        0 &   & 1 & 2 \\
          & 0 & 2 & 1 \\
        \hline
        0 & 1 & 2 & \\
        1 & 0 & & 2 \\
        2 & & 0 & 1 \\
          & 2 & 1 & 0 \\
        \hline
        2 & 0 & 1 & \\
        0 & 2 & & 1 \\
        1 & & 2 & 0 \\
        & 1 & 0 & 2 \\ \hline
      \end{tabular}
      \end{center}
      \caption{Structure forced on $\CAFE(12;G)$ by solitary pairs. \label{cafe1}}
    \end{table}
    The only pairs not covered in this array are the 18 pair of repeated symbols in each pair of columns. Six of these are forbidden by $G$, but the remaining twelve must be covered in the $\CAFE$. 
    The pairs
    \[
      \left \{
        \{v_{0,2},v_{2,2}\},\{v_{1,2},v_{3,2}\}, \{v_{0,1},v_{3,1}\},\{v_{1,1},v_{2,1}\}  
      \right \}
      \]
      can only be covered in the first block of four rows. Thus the first block can only be completed in two possible ways:
      \begin{align*}
         \begin{array}{|cccc|}
        \hline
        1 & 2 & 0 & \textcolor{blue}{2}  \\
        2 & 1 & \textcolor{blue}{2}  & 0 \\
        0 & \textcolor{blue}{1}  & 1 & 2 \\
        \textcolor{blue}{1}  & 0 & 2 & 1 \\
           \hline
           \end{array} &&
         \begin{array}{|cccc|}
        \hline
        1 & 2 & 0 & \textcolor{blue}{1}  \\
        2 & 1 & \textcolor{blue}{1}  & 0 \\
        0 & \textcolor{blue}{2}  & 1 & 2 \\
        \textcolor{blue}{2}  & 0 & 2 & 1 \\
           \hline
           \end{array} 
      \end{align*}
      The pairs
    \[
      \left \{
        \{v_{0,2},v_{1,2}\},\{v_{2,2},v_{3,2}\}, \{v_{0,0},v_{3,0}\},\{v_{1,0},v_{2,0}\}  
      \right \}
      \]
      can only be covered in the second block of four rows. Thus the second block can only be completed in two ways:
      \begin{align*}
         \begin{array}{|cccc|}
        \hline
        0 & 1 & 2 & \textcolor{blue}{1}  \\
        1 & 0 & \textcolor{blue}{1}  & 2 \\
        2 & \textcolor{blue}{0}  & 0 & 1 \\
        \textcolor{blue}{0}  & 2 & 1 & 0 \\
           \hline
           \end{array} &&
         \begin{array}{|cccc|}
        \hline
        0 & 1 & 2 & \textcolor{blue}{0}  \\
        1 & 0 & \textcolor{blue}{0}  & 2 \\
        2 & \textcolor{blue}{1}  & 0 & 1 \\
        \textcolor{blue}{1}  & 2 & 1 & 0 \\
           \hline
           \end{array} 
      \end{align*}
      The pairs
    \[
      \left \{
        \{v_{0,1},v_{1,1}\},\{v_{2,1},v_{3,1}\}, \{v_{0,0},v_{2,0}\},\{v_{1,0},v_{3,0}\}  
      \right \}
      \]
      can only be covered in the last block of four rows. Thus the third block can only be completed in two ways:
      \begin{align*}
         \begin{array}{|cccc|}
        \hline
        2 & 0 & 1 & \textcolor{blue}{0}  \\
        0 & 2 & \textcolor{blue}{0}  & 1 \\
        1 & \textcolor{blue}{2}  & 2 & 0 \\
        \textcolor{blue}{2}  & 1 & 0 & 2 \\
           \hline
           \end{array} &&
         \begin{array}{|cccc|}
        \hline
        2 & 0 & 1 & \textcolor{blue}{2}  \\
        0 & 2 & \textcolor{blue}{2}  & 1 \\
        1 & \textcolor{blue}{0}  & 2 & 0 \\
        \textcolor{blue}{0}  & 1 & 0 & 2 \\
           \hline
           \end{array} 
        \end{align*}
        
Since the automorphism group of $G$ is transitive on the parts of the vertex partition, if there is a uniform column in $\CAFE(12;G)$ it can be assumed to be the first column. Cell $(3,0)$ in the $\CAFE$ can either be a 1 or 2, cell $(7,0)$ can be a 0 or 2 and cell $(11,0)$ can be a 0 or a 1.  If the first column is uniform then this forces the array to be one of the two shown in Table~\ref{cafe2}. 
\begin{table}
  \begin{align*}
  \begin{array}{|cccc|}
    \hline
    1 & 2 & 0 &   \\
    2 & 1 &   & 0 \\
    0 &   & 1 & 2 \\
    1 & 0 & 2 & 1 \\
    \hline
        0 & 1 & 2 & \\
        1 & 0 & & 2 \\
        2 & & 0 & 1 \\
        2 & 2 & 1 & 0 \\
        \hline
        2 & 0 & 1 & \\
        0 & 2 & & 1 \\
        1 & & 2 & 0 \\
        0 & 1 & 0 & 2 \\ \hline
  \end{array} &&
                 \begin{array}{|cccc|}
        \hline
        1 & 2 & 0 &   \\
        2 & 1 &   & 0 \\
        0 &   & 1 & 2 \\
        2 & 0 & 2 & 1 \\
        \hline
        0 & 1 & 2 & \\
        1 & 0 & & 2 \\
        2 & & 0 & 1 \\
        0 & 2 & 1 & 0 \\
        \hline
        2 & 0 & 1 & \\
        0 & 2 & & 1 \\
        1 & & 2 & 0 \\
        1 & 1 & 0 & 2 \\ \hline
                 \end{array}
  \end{align*}
      \caption{Two possible partial $\CAFE(12;G)$ with column 0 uniform. \label{cafe2}}
    \end{table}
    The second is isomorphic to the first under the left action of $g_2 g_1 g_2 g_1 g_2^2$ composing from right to left and a suitable permutation of the rows. So we only consider the first array.  Checking the remaining pairs which must be covered there is only one way to complete the array, shown in Table~\ref{cafe3}.
    \begin{table}
      \begin{center}
      \begin{tabular}{|cccc|}
        \hline
        1 & 2 & 0 & 2 \\
        2 & 1 & 2 & 0 \\
        0 & 1 & 1 & 2 \\
        1 & 0 & 2 & 1 \\
        \hline
        0 & 1 & 2 & 0 \\
        2 & 0 & 0 & 1 \\
        1 & 0 & 2 & 2 \\
        2 & 2 & 1 & 0 \\
        \hline
        0 & 2 & 1 & 1 \\
        1 & 1 & 2 & 0 \\
        2 & 0 & 1 & 0 \\
        0 & 1 & 0 & 2 \\ \hline
      \end{tabular}
      \end{center}
      \caption{Unique $\CAFE(12;G)$ with column 0 uniform. \label{cafe3}}
    \end{table}

There are $2^3 = 8$ possible choices to complete the three blocks of four rows.  Each one of them yields a $\CAFE(12;G)$ that has a single uniform column, so, up to isomorphism, there is a unique non-uniform $\CAFE(12;G)$.
    \end{proof}
    
Stevens and Meagher formulated their conjecture on uniformity of covering arrays in 2005.
\begin{conjecture} \cite[Conjecture
1]{MR2156344} \label{conj:uca}
If there exists a $\CA(N,k,v)$ then there also exists a uniform $\CA(N,k,v)$. 
\end{conjecture}
 With the most current computational techniques to determine as many optimal covering array as possible, the conjecture is still consistent with all known evidence \cite{KMNNOS_18}. In this note we have shown that the conjecture does not hold for the restricted covering arrays avoiding forbidden edges even when the constraints are as symmetric as possible on both the symbols in the array and the constraints.  It would be interesting to look at the uniformity properties of other generalizations of covering arrays.  Covering arrays on vertex transitive graphs would be of interest \cite{MR2156344}.

%\nocite{MR1178497}

%\printbibliography
\bibliographystyle{abbrvnat}
\bibliography{./refs.bib}

% \bib, bibdiv, biblist are defined by the amsrefs package.
\begin{bibdiv}
\begin{biblist}

\bib{colbourn_combinatorial_2004}{article}{
      author={Colbourn, Charles~J.},
       title={Combinatorial aspects of covering arrays},
        date={2004},
     journal={Matematiche (Catania)},
      volume={59},
      number={1-2},
       pages={125\ndash 172 (2006)},
}

\bib{danziger_covering_2009}{article}{
      author={Danziger, Peter},
      author={Mendelsohn, Eric},
      author={Moura, Lucia},
      author={Stevens, Brett},
       title={Covering arrays avoiding forbidden edges},
        date={2009},
        ISSN={0304-3975},
     journal={Theoret. Comput. Sci.},
      volume={410},
      number={52},
       pages={5403\ndash 5414},
}

\bib{hartman_software_2005}{incollection}{
      author={Hartman, Alan},
       title={Software and hardware testing using combinatorial covering
  suites},
        date={2005},
   booktitle={Graph theory, combinatorics and algorithms},
   publisher={Springer},
     address={New York},
       pages={237\ndash 266},
}

\bib{hedayat_orthogonal_1999}{book}{
      author={Hedayat, A.~S.},
      author={Sloane, N. J.~A.},
      author={Stufken, J.},
       title={Orthogonal arrays},
      series={Springer {Series} in {Statistics}},
   publisher={Springer-Verlag},
     address={New York},
        date={1999},
        ISBN={0-387-98766-5},
}

\bib{katona_two_1973}{article}{
      author={Katona, G. O.~H.},
       title={Two applications (for search theory and truth functions) of
  {Sperner} type theorems},
        date={1973},
     journal={PMH},
      volume={3},
       pages={19\ndash 26},
}

\bib{kleitman_families_1973}{article}{
      author={Kleitman, D.~J.},
      author={Spencer, J.},
       title={Families of \$k\$-{Independent} {Sets}},
        date={1973},
     journal={DM},
      volume={6},
       pages={255\ndash 262},
}

\bib{KMNNOS_18}{misc}{
      author={Kokkala, J.~I.},
      author={Meagher, K.},
      author={Naserasr, R.},
      author={Nurmela, P. R.~J., K. J.~\"{O}sterg{\aa}rd},
      author={Stevens, B.},
       title={On the structure of small strength-2 covering arrays},
         how={submitted to J. Combin. Des.},
        date={2018},
        note={submitted to J. Combin. Des.},
}

\bib{lawrence_survey_2011}{article}{
      author={Lawrence, Jim},
      author={Kacker, Raghu~N.},
      author={Lei, Yu},
      author={Kuhn, D.~Richard},
      author={Forbes, Michael},
       title={A {Survey} of {Binary} {Covering} {Arrays}},
    language={en-US},
        date={2011-04},
        ISSN={1077-8926},
     journal={The Electronic Journal of Combinatorics},
      volume={18},
      number={1},
       pages={84},
}

\bib{MR2156344}{article}{
      author={Meagher, Karen},
      author={Stevens, Brett},
       title={Covering arrays on graphs},
        date={2005},
        ISSN={0095-8956},
     journal={J. Combin. Theory Ser. B},
      volume={95},
      number={1},
       pages={134\ndash 151},
}

\bib{ostergaa_rd_disproof_2003}{article}{
      author={{\"O}sterg{\textbackslash}aa~rd, Patric R.~J.},
       title={Disproof of a conjecture on the existence of balanced optimal
  covering codes},
        date={2003},
        ISSN={0018-9448},
     journal={IEEE Trans. Inform. Theory},
      volume={49},
      number={2},
       pages={487\ndash 488},
}

\bib{ostergaa_rd_optimal_2013}{article}{
      author={{\"O}sterg{\textbackslash}aa~rd, Patric R.~J.},
       title={On optimal binary codes with unbalanced coordinates},
        date={2013},
        ISSN={0938-1279},
     journal={Appl. Algebra Engrg. Comm. Comput.},
      volume={24},
      number={3-4},
       pages={197\ndash 200},
}

\end{biblist}
\end{bibdiv}

\end{document}